\documentclass[letterpaper, 10 pt, conference]{ieeeconf}  % Comment this line out
\IEEEoverridecommandlockouts                              % This command is only
\usepackage[utf8]{inputenc}
\usepackage{anysize}
\usepackage{graphicx}
\usepackage{amsmath} 
\usepackage{amsfonts}
\usepackage{float}  
\usepackage{xfrac}
\usepackage{amssymb}
\usepackage{upgreek}
\usepackage{color,soul}
\usepackage{paralist}

\newtheorem{mydef}{Definition}
\newtheorem{mytheo}{Theorem}
\newtheorem{myprop}{Proposition}

\title{Non-linear Gradient Algorithm for Parameter Estimation: Extended version}

\author{
	Juan G. Rueda-Escobedo and Jaime A. Moreno% <-this % stops a space
		\thanks{J.~G. Rueda-Escobedo and J.~A. Moreno are with Eléctrica y Computación, Instituto de Ingeniería, Universidad Nacional Autónoma de 
			México, 04510 México D.F., Mexico, JRuedaE@iingen.unam.mx; JMorenoP@ii.unam.mx}
}

\begin{document}

\maketitle
\thispagestyle{empty}
\pagestyle{empty}

%%%%%%%%%%%%%%%%%%%%%%%%%%%%%%%%%%%%%%%%%%%%%%%%%%%%%%%%%%%%%%%%%%%%%%%%%%%%%%%%
\begin{abstract}
	Gradient algorithms are classical in adaptive control and parameter estimation. For instantaneous quadratic cost functions they lead to a linear 
	time-varying dynamic system that converges exponentially under persistence of excitation conditions. In this paper we consider (instantaneous) 
	non-quadratic cost functions, for which the gradient algorithm leads to non-linear (and non Lipschitz) time-varying dynamics, which are homogeneous 
	in the state. We show that under persistence of excitation conditions they also converge globally, uniformly and asymptotically. Compared to the 
	linear counterpart, they accelerate the convergence and can provide for finite-time or fixed-time stability.
\end{abstract}

%%%%%%%%%%%%%%%%%%%%%%%%%%%%%%%%%%%%%%%%%%%%%%%%%%%%%%%%%%%%%%%%%%%%%%%%%%%%%%%%
\section{Introduction}
	This work is an extended version of \cite{RuedaMoreno15}. In this paper the proof of the claims in \cite{RuedaMoreno15} are presented in the 
	appendices. This proof are omitted in the first work because of its length, instead, the space is used to discuss and clarify the results.\\
	
	A classical linear parametric model is given by $y(t):=u^T(t)\theta_0$, where $\theta_0\in\mathbb{R}^n$ are the parameters, $u(t)\in\mathbb{R}^n$ is 
	the regressor and $y(t)\in\mathbb{R}$ is the measured signal. Output $y(t)$ along with an estimate of the parameters $\hat{\theta}(t)$ is used to 
	build the output estimation error $e(t):=u^T(t)\hat{\theta}(t)-y(t)$, which can be rewritten as $e(t)=u^T(t)x(t)$, where $x(t):=\hat{\theta}(t)-
	\theta_0$ is the parameter estimation error. The aim is to use $e$ to drive $x$ to zero. Since Persistence of Excitation (PE) of $u(t)$ is equivalent 
	to the Uniform and Complete Observability of the associated linear dynamical system $\dot{\theta}_0 = 0,\, y(t)=u^T(t)\theta_0$ \cite{Anderson_77}, 
	\cite{Mor_Nar_77}, it is a necessary condition to assure uniform and robust convergence of any algorithm. In particular, for the Linear Gradient 
	Descent and the Recursive Least Square Methods PE is a necessary and sufficient condition for exponential convergence \cite{Anderson_77}, 
	\cite{Mor_Nar_77}.\\
	
	In fact, the use of correction terms linear in $e$ cannot provide for convergence faster than exponential. So, if accelerated convergence is desired, 
	algorithms using nonlinear correcting terms are required. In recent years, the use of homogeneous systems and homogeneous correction terms for 
	control and observation purposes have been very successful in providing finite time and fixed time convergence \cite{BhaBer05,AndPra08}. Furthermore, 
	homogeneous higher order sliding modes (HOSM) provide for discontinuous correction terms, which provide not only finite time convergence but also 
	insensitivity to (matched, bounded) perturbations \cite{Lev05, Mor11, BerEfi14}. Homogeneity has been important for these results, since it provides 
	useful properties: e.g.  local asymptotic stability is equivalent to global finite time stability for systems of negative homogeneity degree 
	\cite{LiapFunc_Bacciotti_Rosier, Lev05, BerEfi13}.\\
		
	Motivated by these results we propose in this work non-linear estimation algorithms, with non-linear correction terms in $e$, that lead to 
	time-varying dynamical systems. The proposed schemes can be obtained as the negative gradient of an instantaneous convex non smooth function of the 
	output estimation error. The resulting (error) systems are homogeneous (or homogeneous in the bi-limit) in the estimation error $x$ 
	\cite{LiapFunc_Bacciotti_Rosier}, \cite{Peut_Aeyels_99}, but time-varying. Unlike time invariant homogeneous systems, time-varying systems do not 
	possess such strong properties. For example, neither local asymptotic stability implies global asymptotic stability, nor negative homogeneity degree 
	implies finite time convergence \cite{Peut_Aeyels_99}. Some local asymptotic stability results for time-varying systems homogeneous in the state have 
	been obtained, when the homogeneity degree is zero \cite{MCloskey_97}, and for positive homogeneity degree \cite{Peut_Aeyels_99}, using 
	averaging techniques.\\
	
	We show in this paper, in straightforward a direct manner, that the proposed algorithms converge globally and asymptotically under the PE condition 
	of the regressor. When $n=1$, the algorithm is able to converge in finite-time and it is also able to estimate a time-varying parameter. Moreover, 
	adding homogeneous terms of positive degree the estimation error converges in fixed-time, i.e. the convergence time is upper bounded by a constant 
	independent of the initial estimation error. For $n>1$ global, uniform and asymptotic stability can be assured and acceleration for large initial 
	conditions can be obtained. 
	
\section{Motivation and Problem Statement}
	The estimation problem may be regarded as a minimization problem. In this context a common cornerstone is to set a convex cost function of the output 
	estimation error. In this work we choose the following structure for the cost function
	\begin{align}
		J(\hat{\theta})_p:=\frac{1}{p+1}\big|u^T(t)\hat{\theta}(t)-y(t)\big|^{p+1}.\nonumber	
	\end{align}	 
	Here the exponent $p>0$ is a parameter to be chosen. The term $-\frac{\partial}{\partial \hat{\theta}}J(\hat{\theta})_p$ shows the direction in which 
	the parameter estimated needs to change. With this idea in mind we propose the following algorithm
	\begin{small}
		\begin{align}
			\dot{\hat{\theta}}(t)&:=-\frac{\partial}{\partial \hat{\theta}} J(\hat{\theta})_p\nonumber\\
			&=-\big|u^T(t)\hat{\theta}(t)-y(t)\big|^p\text{sign}\big(u^T(t)\hat{\theta}(t)-y(t)\big)u(t).\nonumber
		\end{align}
	\end{small}
	For the sake of readability let us define $\lceil w\rfloor^p:=|w|^p\text{sign}(w)$. With this convention the algorithm is 
	rewritten as
	\begin{align}
		\label{Alg}
		\dot{\hat{\theta}}(t)=-\lceil u^T(t)\hat{\theta}(t)-y(t)\rfloor^pu(t).
	\end{align}
	We denote as \textit{composite algorithm} an algorithm that results from adding vector fields of the form in \eqref{Alg} to avail of the dynamics 
	traits associated to each $p_i$. This leads to
	\begin{align}
		\label{CompAlg}
		\dot{\hat{\theta}}(t)=-\left(\sum_{i=1}^h\lceil u^T(t)\hat{\theta}(t)-y(t)\rfloor^{p_i}\right)u(t).
	\end{align}	
	To analyse the convergence of the algorithms to $\theta_0$ the dynamics of the estimation error $x(t)=\hat{\theta}(t)-\theta_0$
	is needed. Hence it is necessary to compute the time derivative of $x(t)$. In the case of the \textit{single algorithm} in \eqref{Alg} this error 
	dynamics is as follow
	\begin{align}
		\label{ErrDyn00}
		\dot{x}(t)=-\lceil u^T(t)x(t)\rfloor^pu(t).
	\end{align}
	Considering $f(t,x)=-\lceil u^T(t)x\rfloor^pu(t)$ and replacing the argument $x$ by $\epsilon x$, with $\epsilon>0$, we have the following relation 
	$f(t,\epsilon x)=\epsilon^{p-1}\epsilon f(t,x)$. Taking this into consideration and following Definition 1 in \cite{Peut_Aeyels_99}, we can said that 
	the system \eqref{ErrDyn00} is homogeneous with homogeneity degree $p-1$. When $0<p<1$ the homogeneity degree is negative; for $p=1$ it is zero; and 
	for $p>1$ the homogeneity degree is positive. Now, the convergence analysis for this class of algorithms reduces to establish the stability and the 
	attractivity of the origin of a time-variant homogeneous system.\\
	 
	Repeating the same analysis for the composite algorithm the error dynamics is
	\begin{align}
		\label{ErrDyn01}
		\dot{x}(t)=-\left(\sum_{i=1}^h\lceil u^T(t)x(t)\rfloor^{p_i}\right)u(t).
	\end{align}
	Please, notice that this system is no longer homogeneous.\\
	
	A quick stability check with a quadratic Lyapunov function $V(x):=\frac{1}{2}x^Tx$ yields to global uniform stability (GUS) of the 
	origin of \eqref{ErrDyn00} and \eqref{ErrDyn01}. The time derivative of $V$ is presented for both cases in order to show its negative 
	semi-definitiveness:
	\begin{subequations}
	\begin{small}
		\begin{align}
			\label{dotVa}
			\text{along \eqref{ErrDyn00}}&		&\dot{V}(t)=-|u^T(t)x(t)|^{p+1}\leq0	,\\
			\label{dotVb}
			\text{along \eqref{ErrDyn01}}&		&\dot{V}(t)=-\sum_{i=1}^h|u^T(t)x(t)|^{p_i+1}\leq0.
		\end{align}
	\end{small}
	\end{subequations}
	Notice that the term $u^T(t)x(t)$ can vanish outside the set $\{t|u(t)=0\}\cup\{t|x(t)=0\}$. To assert the uniform asymptotic stability (UAS) of 
	the origin of \eqref{ErrDyn00} and \eqref{ErrDyn01} $u(t)$ needs to be PE. In \cite{Nar_Ann_87} a convenient description of the Persistent
	Excitation is given. This description is expressed as a lower bound of the integral of the regressor and is found convenient because it fits well in 
	the study of the estimation-error convergence. For convenience the definition is reproduced below
	\begin{mydef}
		\label{PE}
		Let $u(t):\mathbb{R}_+\to\mathbb{R}^n$ be a piecewise continuous function. It is said that $u$ is of PE if there exist $T>0$ and $\epsilon>0$ 
		such that
		\begin{align}
			\frac{1}{T}\int_{t}^{t+T}|u^T(s)w|\text{d}s\geq \epsilon,\nonumber
		\end{align}
		for all $w\in\mathbb{R}^n$ with $\|w\|=1$.
	\end{mydef}
	\hfill $\triangle$\\
	To conclude this section, we present some definitions of time-varying systems.
	\begin{mydef}
		Let a time-varying system be represented by $\dot{x}(t)=f(x(t),t)$ where $f(0,t)=0$ for all $t$. Let $\Omega$ be a connected open subset of 
		$\mathbb{R}^n$, such that $0\in\Omega$. The point $x=0$ is 
		\begin{itemize}
			\item	Uniformly finite-time stable (UFTS) if it is uniformly stable and for any $x_0\in\Omega$ there exist $0\leq \textbf{T}(x_0)<+\infty$ 
					such that $x(t,t_0,x_0)=0$ for all $t\geq t_0+\textbf{T}(x_0)$. Also, if $\Omega=\mathbb{R}^n$ then $x=0$ is said to be globally 
					uniformly finite-time stable (GUFTS).
			\item	Uniformly fixed-time stable (UFxTS) if it is GUFTS and exist $\bar{\textbf{T}}<+\infty$ such that 
					$\bar{\textbf{T}}\geq \textbf{T}(x_0)$ for every $x_0\in\mathbb{R}^n$.
		\end{itemize}
		\hfill $\triangle$
	\end{mydef} 
		
\section{Main Result}
	In this section the stability of \eqref{ErrDyn00} and \eqref{ErrDyn01} is presented. Two cases are recognized: \begin{inparaenum}[(i)] \item when only 
	one parameter needs to be estimated (scalar case) and \item when are more than one (vector case)\end{inparaenum}. This division is done because the 
	results in the vector case do not reflect certain phenomena that occurred in the scalar one.\\
	The proof of the theorems can be found in the Appendices.
	
	\subsection{Scalar Case}
		The results in the scalar case are stronger than in the vector one due to the fact that the product $u^T(t)x(t)$ can only be zero if one
		or both of the variables are zero. Persistent excitation prevents $u$ to stay in zero or to exhibit a growing dwelling time in zero. This implies
		that, for staying in zero, $x$ needs to be zero. When the exponent $p$ is chosen in the interval $[0,1)$ the algorithm converges in uniform 
		finite-time and a bound of the convergence time is given as an integer multiple of the persistent excitation period $T$. The next statement 
		summarizes this discussion.
		\begin{mytheo}
			\label{TheoScalar1}
			Let $u(t)$ be a piece-wise continuous function of $t$ and of PE. Let  $0\leq p<1$ and $n=1$, then the origin of \eqref{ErrDyn00} is globally 
			uniformly finite-time stable. An upper bound of the convergence time is
			\begin{align}
				\label{time0}
				\textbf{T}(x_0)\leq\left\lceil \frac{|x(t_0)|^{1-p}}{(1-p)T\epsilon^{p+1}}\right\rceil\cdot T
			\end{align}
			where $x(t_0)$ is the estimation error at the initial time; $T$ and $\epsilon$ are as in Definition \ref{PE} and $\lceil \cdot\rceil$ 
			denotes the ceiling function.
		\end{mytheo}		
		\hfill $\triangle$\\
		When $p=1$ the linear case is obtained and its properties are well know \cite{Anderson_77,Mor_Nar_77}, for this is left out of the 
		discussion. For $p>1$ only global uniform asymptotic stability (GUAS) can be asserted but other interesting property arise. No matter how large 
		the initial error is, the time $\tau(c)$ need to reach a smaller level set $V(x)=c$ is only function of $c$. This is referred to as 
		\textit{escape from infinite in finite time uniformly in} $t$. We gather these results in the following theorem.
		\begin{mytheo}
			\label{TheoScalar2}
			Let $u(t)$ be a piece-wise continuous function of $t$ and of PE, $p>1$ and $n=1$, then the origin of \eqref{ErrDyn00} is globally uniformly 
			asymptotically stable. Furthermore, the time needed to escape from infinity to a compact region $V(x)\leq c$ is bounded by
			\begin{align}
				\label{time2}
				\tau(c)\leq\left\lceil \frac{1}{2^{\frac{p-1}{2}}(p-1)T\epsilon^{p+1}}\cdot\frac{1}{c^{\frac{p-1}{2}}}\right\rceil\cdot T
			\end{align}
			\hfill $\triangle$
		\end{mytheo}
		The study of the convergence of the composite algorithm can be done using the previous results. For the scalar case the stability can be asserted 
		via a Comparison Lemma for differential inequalities \cite{NonLin_Sys_Khalil}. In short, the trajectories of the system \eqref{ErrDyn01} are  
		below the trajectories corresponding to each of the single algorithms with every isolated exponent $p_i$. This can be seen from \eqref{dotVb} 
		which can be rewritten as
		\begin{align}
			\dot{V}(t)&=-\sum_{i=1}^h 2^{\frac{p_i+1}{2}}|u(t)|^{p_1+i}V^{\frac{p_i+1}{2}}(t)\nonumber\\
				&\leq -2^{\frac{p_i+1}{2}}|u(t)|^{p_1+i}V^{\frac{p_i+1}{2}}(t),\ i=1,2,\cdots,h.\nonumber
		\end{align}		 
		An important case occurs when at least one exponent is in $[0,1)$ and another one is greater than one. For this case the time needed 
		for the algorithm to converge to $\theta_0$ is independent of the initial estimation error and the initial time if $u(t)$ is of PE. This is
		summarized in the following theorem.
		\begin{mytheo}
			\label{TheoScalar3}
			Let $u(t)$ be a piece-wise continuous function of $t$ and of PE, also $n=1$. Consider the system \eqref{ErrDyn01} and let 
			$\textbf{P}:=\{p_1,p_2,\cdots,p_h\}$ be the set of the exponents. Denote $p_m$ the minimum element in \textbf{P} and $p_M$ the maximum. Assume 
			that $0\leq p_m<1$ and $p_M>1$. Let $\textbf{P}_m$ be the subset of the exponent smaller than one and $\textbf{P}_M$ the subset of the 
			exponents greater than one. Define $\bar{p}\in \textbf{P}_M$ as the exponent which maximizes $(p_i-1)\epsilon^{p_i+1}$ and 
			$\underline{p}\in\textbf{P}_m$ the exponent which maximizes $(1-p_i)\epsilon^{p_i+1}$, then an upper bound of the time needed to converge to 
			zero is
			\begin{align}
				\bar{\textbf{T}}\leq\left(\left\lceil \frac{1}{(1-\underline{p})T\epsilon^{\underline{p}+1}}\right\rceil
					+\left\lceil \frac{1}{2^{\frac{\bar{p}-1}{2}}(\bar{p}-1)T\epsilon^{\bar{p}+1}}\right\rceil\right)\cdot T.\nonumber
			\end{align}\hfill $\triangle$
		\end{mytheo}
		Last but no least, a discontinuous algorithm capable of estimating one varying parameter is presented. This algorithm makes use of the 
		regressor sign and the sign of the output estimation error
		\begin{align}
			\label{DiscAlg}
			\dot{\hat{\theta}}(t)&=-L\cdot\text{sign}\big(u(t)\hat{\theta}(t)-y(t)\big)\cdot\text{sign}\big(u(t)\big),
		\end{align}
		with $L>0$. Now assume that the parameter variation is bounded, i.e. $\dot{\theta}\in[-\gamma,\gamma]$, $\gamma\geq 0$. Also it is considered that 
		$u(t)$ is of PE and cannot stay in zero for time intervals, but can cross it. The error dynamics induced by \eqref{DiscAlg} is
		\begin{align}
			\label{DiscErrD}
			\dot{x}(t)\in-L\cdot\text{sign}\big(x(t)\big)+\dot{\theta}(t).
		\end{align}
		which is a differential inclusion \cite{LiapFunc_Bacciotti_Rosier}. By employing $V(x)=\frac{1}{2}|x|^2$ as Lyapunov Function its derivative along 
		the trajectories of \eqref{DiscErrD} is
		\begin{align}
			\dot{V}(t)&=-L|x(t)|-\dot{\theta}x(t)\nonumber\\
				&=-\big(L+\dot{\theta}\cdot\text{sign}(x(t))\big)|x(t)|,\nonumber
		\end{align}
		which is negative if $L>\gamma$ and $u(t)\neq 0$. If $u(t)$ stay in zero the track of $\theta(t)$ is lost but can be recovered later if $u$ is of 
		PE. To guarantee exact tracking, $u(t)$ cannot stay in zero.
		
	\subsection{Vector Case}
		When $\theta_0$ is a vector the set where $u^T(t)x(t)$ is zero grows. This changes the general behaviour of our algorithms. Only
		GUAS can be guaranteed in general with persistent excitation. Also, the discontinuous algorithm which result of selecting $p=0$ does not 
		converge. In the latter case, a signal of PE can be constructed for which the output estimation error $e(t)$ becomes zero in finite time but 
		$x(t)$ does not reach the origin. The stability properties of the origin of \eqref{ErrDyn00} are summarized in the following theorem.
		\begin{mytheo}
		\label{TVecS}
			Let $u(t)$ be a piece-wise continuous function of $t$ and of PE, uniformly bounded by $u_M$, then the origin of \eqref{ErrDyn00} is globally 
			uniformly asymptotically stable for any $p>0$.
		\end{mytheo}
		\hfill $\triangle$\\
		Although only GUAS can be asserted with PE in general, in the following section two classes of signals of PE are presented. One of them 
		guarantees global uniform exponential stability (GUES) and the other GUFTS; in both cases for $0<p<1$. This means that the signals $u(t)$ which 
		can provide UFTS are in a subset of those of PE.\\
		The composite algorithm still works in the vector case but nothing more than GUAS can be claimed. In contrast to the scalar case the stability of 
		the composite vector algorithm cannot be obtained via the Comparison Lemma. The following theorem synthesizes this discussion
		\begin{mytheo}
			\label{TVecC}
			 Let $u(t)$ be a piece-wise continuous function of $t$ and of PE, uniformly bounded by $u_M$, then the origin of \eqref{ErrDyn01} is globally
		 	uniformly asymptotically stable if $p_M\leq p_m+1$	where $p_m$ is the minimum exponent in the set $\{p_1,p_2,\cdots,p_h\}$ and $p_M$ is the 
		 	maximum.
		\end{mytheo}
	 	\hfill $\triangle$\\
		The extra condition regarding the exponents appears due to the way the proof was done and we think it is not intrinsic to the stability.\\
		In the scalar case the Comparison Lemma gives information about the relationship between the trajectories of the single algorithm but in this 
		case nothing can be concluded. In the next section simulation examples are presented for the composite algorithm. In the Figure \ref{Fig04} and
		Figure \ref{Fig05} it can be seen that the Lyapunov function is below of the corresponding one for the single algorithms but this does not need 
		to be true in general.
	
\section{Examples}	
	In this section examples are presented. In the scalar case, simulations of a parameter estimation process are shown, whereas for the vector case also 
	the behaviour of the error is studied for a specific class of signals.	
	\subsection{Scalar case}
		Numerical simulation is performed to illustrate the difference between the classic gradient algorithm and the family presented in this work. 
		For this aim we choose:
		\begin{align}
			\begin{matrix}
				\theta_0=5,	& \hat{\theta}(t_0)=0,	&	u(t)=2\cos(2t).
			\end{matrix}\nonumber
		\end{align}
		Four algorithms were simulated: one with $p_1=\frac{3}{4}$, other with $p_2=1$, a third one with $p_3=\frac{3}{2}$ and a composite one with 
		$p_1$ and $p_3$. Figure \ref{Fig01} and Figure \ref{Fig02} show the behaviour of the Lyapunov function 
		$V(t)=\frac{1}{2}|\hat{\theta}(t)-\theta_0|^2$. In the first figure the initial value of $V(t)$ and its decaying behaviour are presented for all 
		the algorithms. As can be seen there is an ordering in the decay: $\text{Composite}>p=\frac{3}{2}>p=1>p=\frac{3}{4}$, and this is due to the 
		initial value of $V(t_0)>>1$. When the estimation error becomes smaller the order changes as in Figure \ref{Fig02}.\\
		\begin{figure}[h]
			\centering
				\includegraphics[width=2.5in]{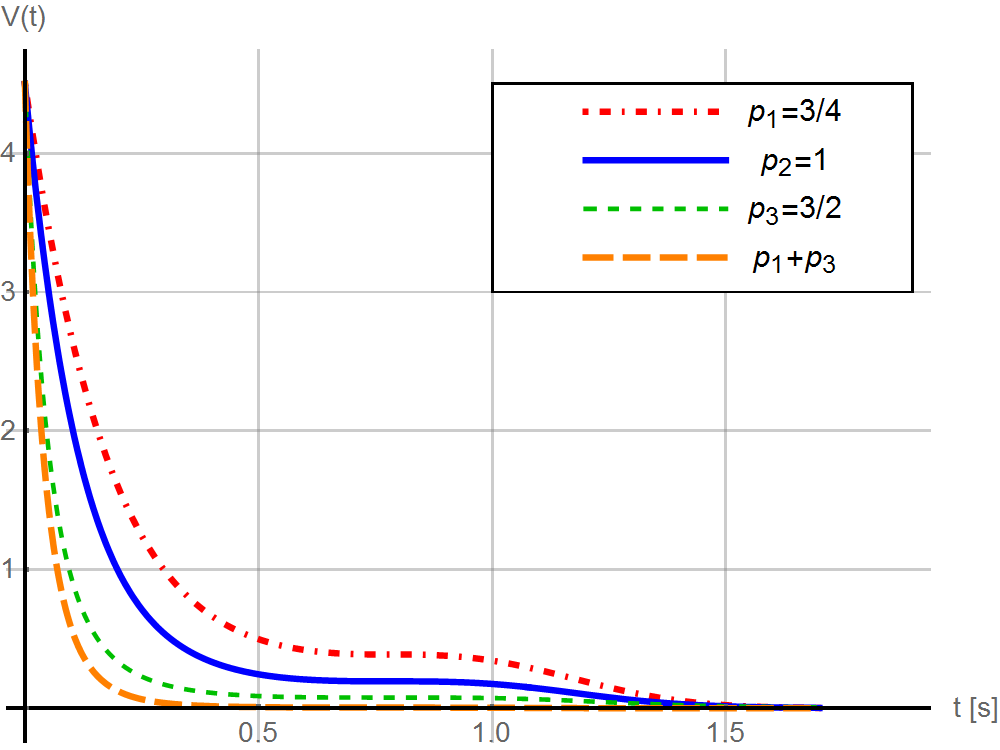}
				\caption{Behaviour of $V(t)$ in the Scalar case (a).}
				\label{Fig01}
		\end{figure}
		In Figure \ref{Fig02} the behaviour is shown when $V(t)<0.014$. In this region the decay order is: 
		$\text{Composite}>p=\frac{3}{4}>p=1>p=\frac{3}{2}$ as is expected from the kind of convergence of every algorithm: 
		$\text{u. fixed-time}>\text{u. finite-time}>\text{u. exponential}>\text{u. asymptotic}$. This order does not change for any future time.\\
		\begin{figure}[h]
			\centering
				\includegraphics[width=2.5in]{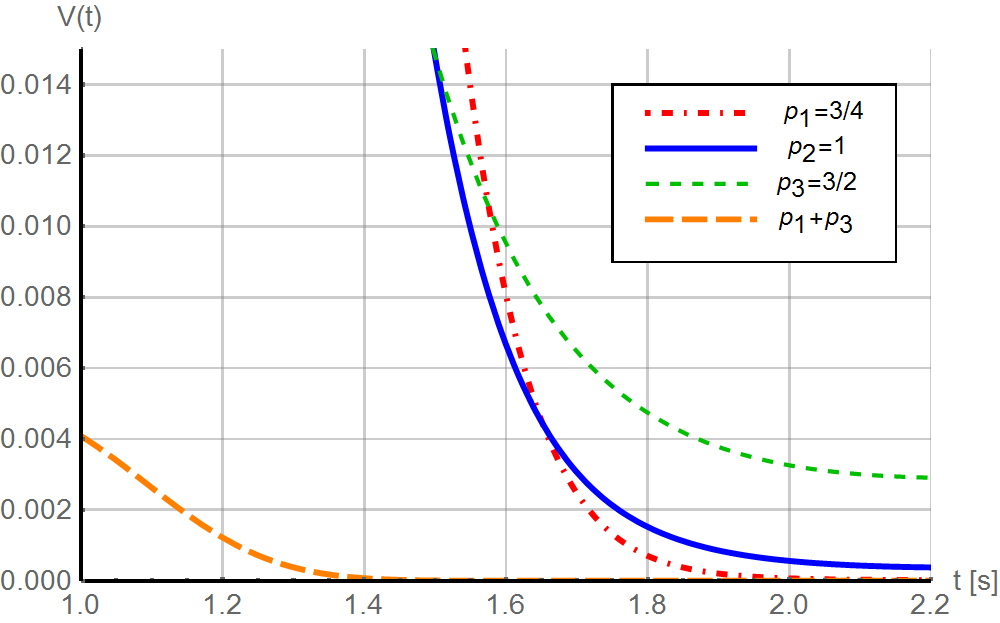}
				\caption{Behaviour of $V(t)$ in the Scalar case (b).}
				\label{Fig02}
		\end{figure}
		Now a estimation of a time-varying parameter is done via \eqref{DiscAlg}. The simulation parameters are:
		\begin{small}		
			\begin{align}
				\begin{matrix}
					\theta(t)=\cos(3t)-4,	&	\hat{\theta}(t_0)=0,	&	u(t)=\cos(t)+1.5,\\
											&	L=3.3.					&
				\end{matrix}\nonumber
	 		\end{align}
		\end{small}
		Since the regressor is non-zero for any time, an exact tracking of the parameter is achieved as shown in Figure \ref{Fig03}. 
		\begin{figure}[h]
			\centering
				\includegraphics[width=2.5in]{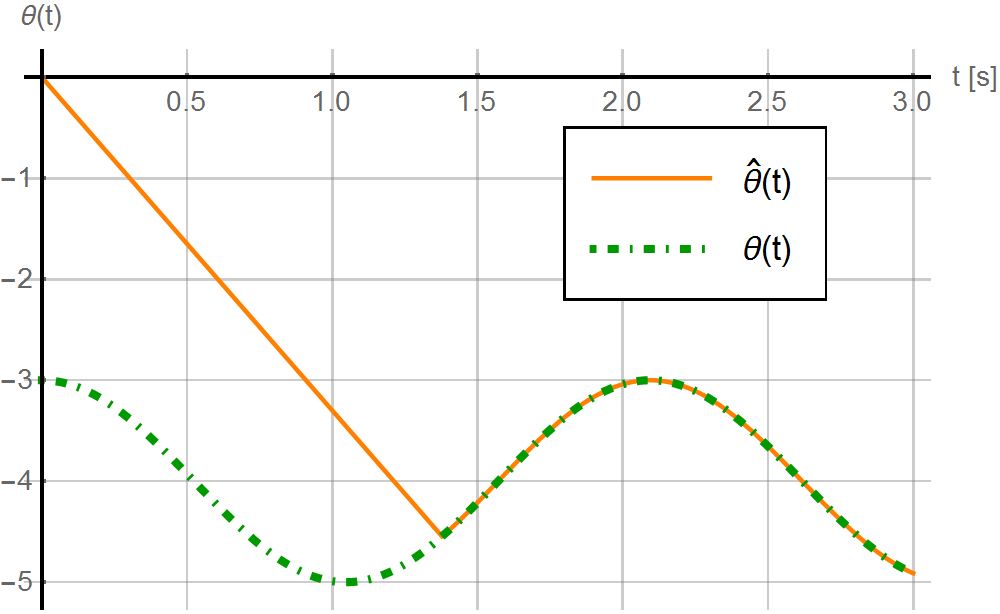}
				\caption{Time-varying parameter estimation process.}
				\label{Fig03}
		\end{figure}
		
	\subsection{Vector case}
		First simulation results are presented for the estimation process of three parameters with \eqref{Alg} and \eqref{CompAlg}. Four algorithms were
		simulated: the first with $p_1=3/4$, the second with $p_2=1$, the third one with $p_3=3/2$ and the last one with $p_1$ and $p_3$. The simulation 
		results are shown in Figure \ref{Fig04} and Figure \ref{Fig05} to compare the behaviour of the algorithms, when different exponents are chosen. 
		Later, the response of the algorithm for $0<p<1$ is found for piecewise constant signal and the class of convergence is established for it.\\
		For the simulation example the next conditions were selected:
		\begin{small}
			\begin{align}
				u(t)&=\begin{bmatrix}2\cos(2t)	&	-\cos(3t)	&	5\cos(5t)\end{bmatrix}^T,\nonumber\\
				\theta_0&=\begin{bmatrix}-3	&	\sqrt{2}	&	4\end{bmatrix}^T,\nonumber\\
				\hat{\theta}(t_0)&=\begin{bmatrix}0	&	0	&	0\end{bmatrix}^T.\nonumber
			\end{align}
		\end{small}
		As in the last section only the plot of $V(t)=\frac{1}{2}(\hat{\theta}(t)-\theta_0)^T(\hat{\theta}(t)-\theta_0)$ is presented. In Figure
		\ref{Fig04} the decay of $V(t)$ is shown at the beginning of the process and the decay order is preserved. Also in Figure \ref{Fig05} the
		change of order that was found for the scalar case is obtained. However we do not see, in general, a change of order for arbitrary parameter 
		values.\\
		\begin{figure}[h]
			\centering
				\includegraphics[width=2.5in]{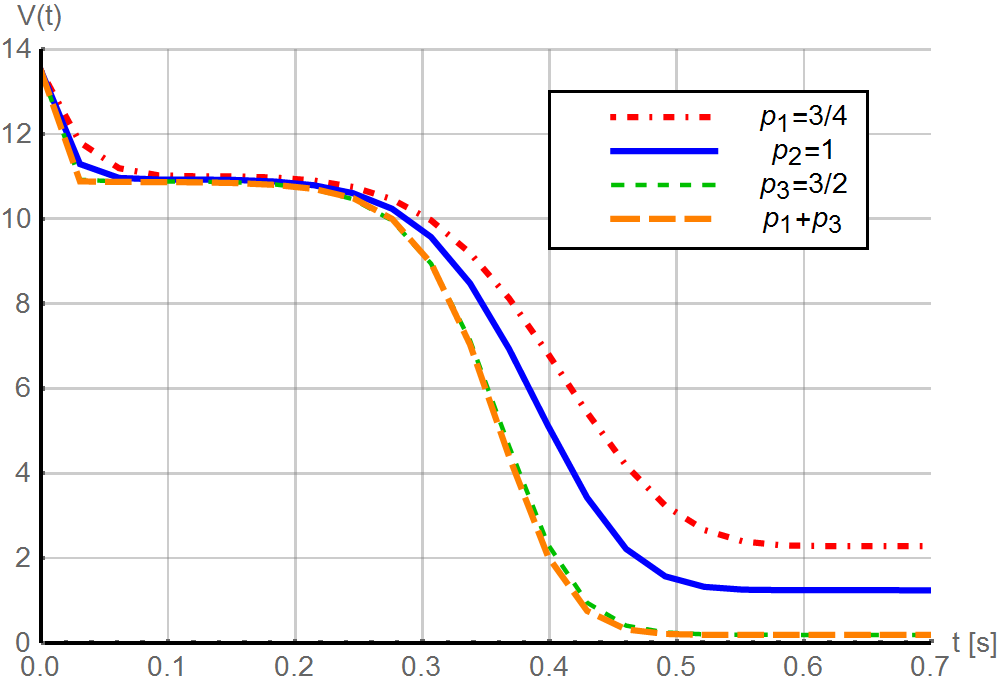}
				\caption{Simulation results for vector case (a).}
				\label{Fig04}
		\end{figure}
		\begin{figure}[h]
			\centering
				\includegraphics[width=2.5in]{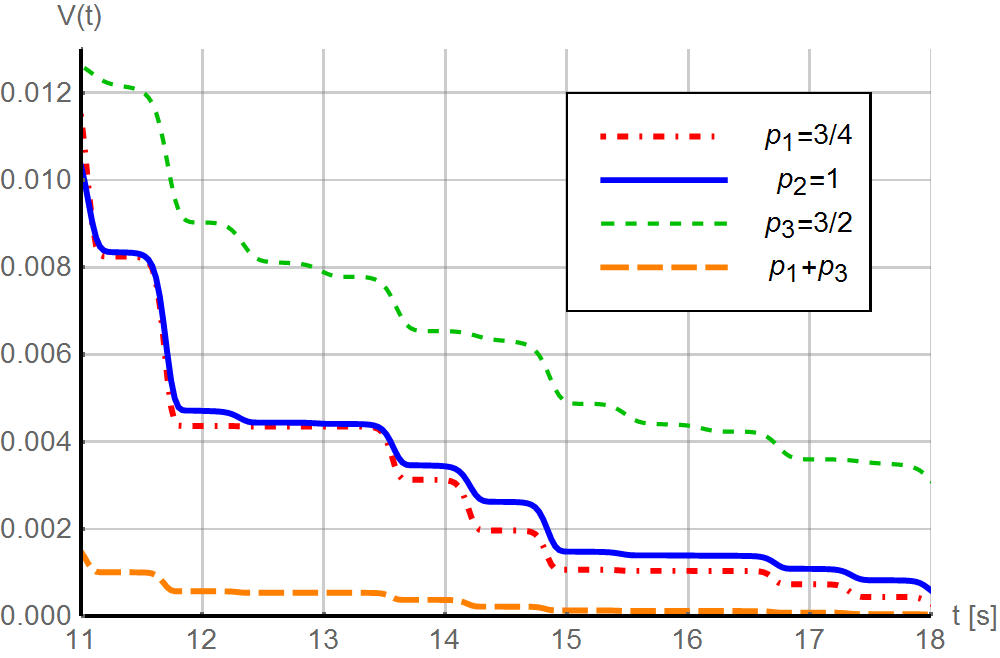}
				\caption{Simulation results vector case (b).}
				\label{Fig05}
		\end{figure}
		Now the response of the system \eqref{ErrDyn00} when $u(t)$ is piecewise constant is analyzed. Let us define a new variable as follows 
		$z(t):=u^T(t)x(t)$, its dynamics is $\dot{z}(t)=u^T(t)\dot{x}(t)=-\|u(t)\|^2\lceil z(t)\rfloor^p$ since $\dot{u}(t)=0$ almost everywhere, where 
		$\|\cdot\|$ denotes the Euclidean norm. For simplicity we assume that the length of the intervals where $u$ remains constant has a constant 
		value $\tau>0$. The solution of $z(t)$ in the interval $t\in[t_1,t_1+\tau]$ with $u(t)=\mu$, $u_M\geq\|\mu\|>0$ and $p>0$, $p\neq 1$ is
		\begin{small}
			\begin{align}
				z(t)=\left(|z(t_1)|^{1-p}-(1-p)\|\mu\|^2(t-t_1)\right)^{\frac{1}{1-p}}\text{sign}(z(t_1)).\nonumber
			\end{align}
		\end{small}
		For $0<p<1$ the expression above holds if $|z(t_1)|^{1-p}\geq (1-p)\|u\|^2(t-t_1)$ and $z(t)=0$ otherwise. The solution of $z(t)$ can be used to 
		find the solution of $x(t)$ by noticing that $\dot{x}(t)=-\lceil z(t)\rfloor^pu$. This is given in the forthcoming expression
		\begin{small}
			\begin{align}
				x(t)=x(t_1)-&\frac{1}{\|\mu\|^2}\mu\Big(\mu^Tx(t_1)-\text{sign}(\mu^Tx(t_1))\times\nonumber\\
					\times&\big(|\mu^Tx(t_1)|^{1-p}-(1-p)\|\mu\|^2(t-t_1)\big)^{\frac{1}{1-p}}\Big).\nonumber
			\end{align}
		\end{small}
		Restricting the analysis for $0<p<1$ it is clear that, if $\tau$ is large enough then $x(t_1+\tau)$ is orthogonal to $\mu(t_1)$
		\begin{footnotesize}
			\begin{align}
				x(t_1+\tau)&=x(t_1)-\frac{1}{\|\mu\|^2}\mu\mu^Tx(t_1)=\Big(\mathbb{I}_n-\frac{1}{\|\mu\|^2}\mu\mu^T\Big)x(t_1).\nonumber
			\end{align}
		\end{footnotesize}
		Fix $\tau$, if $\|x(t_1)\|^{1-p}\leq (1-p)u_m^{p+1}~\tau$ then $x(t_1+\tau)$ becomes orthogonal to $\mu$. This means that there exists a ball 
		centered in zero for which $\tau$ is always large enough to make $x$ orthogonal to any $\mu$. If the sequence $U=\{\mu_i\}_{i=1}^{\infty}$ is 
		chosen to fulfil the notion of persistent excitation for discrete systems in \cite{Lee_Narendra_88}, then the origin is GUAS by Theorem 
		\ref{TVecS} and $x(t)$ can reach any ball centered in zero in finite time, i.e. always reach the ball in where $\tau$ guarantees that $x$ becomes 
		orthogonal to $\mu_i$, and this yields a discrete system which is described by the following difference equation
		\begin{align}
			x_{k+1}&=\Big(\mathbb{I}_n-\frac{1}{\|\mu_k\|^2}\mu_k\mu_k^T\Big)x_k.\nonumber
		\end{align}
		Only assuming PE of the sequence $U$, GUES of the origin can be concluded making an analogue analysis to the one shown in the proof for Theorem 4
		in Appendix \ref{Proof1}  but taking $V_k(x)=x_k^Tx_k$, $\Delta V=V_{k+T}-V_k(x)$ instead of $V(x)$ and $\dot{V}(t)$ respectively. Now take $n$ 
		mutual orthogonal vectors $\{v_1,v_2,\cdots,v_n\}$, $v^T_iv_j=0$, $i\neq j$, and excite the system with them, then $x_{k+n}=0$
		\begin{footnotesize}
			\begin{align}
				x_{k+n}&=\prod_{i=1}^n\left(\mathbb{I}_n-\frac{1}{\|v_i\|^2}v_iv_i^T\right)\cdot x_k\nonumber\\
					&=\left(\mathbb{I}_n-\sum_{i=1}^n\frac{1}{\|v_i\|^2}v_iv_i^T\right)x_k=0.\nonumber
			\end{align}
		\end{footnotesize}
		A sequence constructed with these vectors is of PE and also makes the system GUFTS. This shows that persistent excitation cannot guarantee 
		finite-time convergence in the vector case but it does not forbid it.
\section{Conclusions}
	In this work a parameter estimation technique is presented. With the proposed algorithms we obtained finite-time and fixed-time 
	convergence to the true parameters. However this properties cannot be guaranteed in general. A deep study of the signals that can assert such 
	important properties is still needed.\\
	Even though the algorithms were selected to make the error dynamics homogeneous in the state, this does not help in the analysis. However, the 
	homogeneous non-linearities can enhance the robustness properties of our algorithms w.r.t additive perturbations in comparison with classical
	approaches.

\section*{Acknowledgement}
	The authors thank the financial support from PAPIIT-UNAM (Programa de Apoyo a Proyectos de Investigación e Innovación Tecnológica), project IN113614; 
	Fondo de Colaboración II-FI UNAM, Project IISGBAS-122-2014; CONACyT (Consejo Nacional de Ciencia y Tecnología), project 241171; and CONACyT CVU: 
	491701.

\bibliographystyle{IEEEtran}
\bibliography{IEEEabrv,myBib}

\appendices
\section{Persistent Excitation}
	To prove uniform asymptotic stability rather than uniform stability the persistent excitation in the regressor is needed. For using the property
	adequately the next proposition is developed.
	\begin{myprop}
		\label{PE_2}
		If $u(t)$ is of PE, then the following inequality holds for $T$, $\epsilon$ and $w$ as in Definition \ref{PE} and $p\geq 0$
		\begin{align}
			\int_{t}^{t+T}|u^T(s)w|^{p+1}\text{d}s\geq T\epsilon^{p+1}.\nonumber
		\end{align}
	\end{myprop}
	\begin{proof}
		Applying the Hölder inequality to $f(t)=|u^T(t)w|$ and $g(t)=1$ on the interval $[t,t+T]$ and using $\frac{1}{p+1}$ and $\frac{p}{p+1}$ as 
		Hölder conjugates we obtain
		\begin{small}
			\begin{align}
				T^p\int_{t}^{t+T}|u^T(s)w|^{p+1}\text{d}s\geq \left(\int_t^{t+T}|u^T(s)w|\text{d}s\right)^{p+1},\nonumber
			\end{align}
		\end{small}
		using Definition \ref{PE} leads to
		\begin{align}
			\int_{t}^{t+T}|u^T(s)w|^{p+1}\text{d}s\geq \frac{T^{p+1}}{T^p}\epsilon^{p+1}=T\epsilon^{p+1}.\nonumber
		\end{align}
	\end{proof}
	This derivation from the PE is done in the aim of easily present the proof of the theorems.
\section{Proof of Theorems \ref{TheoScalar1} to \ref{TVecC}}
	\label{TheoProofs}
	\subsection{Proof of Theorem \ref{TheoScalar1}}
		For $n=1$ the term $|u(t)^Tx(t)|$ can be rewritten as $|u(t)|\cdot|x(t)|=\sqrt{2}|u(t)|V^{\frac{1}{2}}(t)$ replacing this in \eqref{dotVa} 
		yields
		\begin{align}
			\dot{V}(t)=-2^{\frac{p+1}{2}}|u(t)|^{p+1}V^{\frac{p+1}{2}}(t).\nonumber
		\end{align}
		Solving the differential equation we have
		\begin{small}
		\begin{align}
			\label{SolV}
			V(t)=\left(V^{\frac{1-p}{2}}(t_0)-2^{\frac{p-1}{2}}(1-p)\int_{t_0}^t|u(s)|^{p+1}\text{d}s\right)^{\frac{2}{1-p}}.
		\end{align}
		\end{small}
		This solution is valid for $p\geq 0$ and $p\neq 1$. For $0\leq p<1$ the solution exists if $V^{\frac{1-p}{2}}(t_0)\geq $ 
		$2^{\frac{p-1}{2}}(1-p)\int_{t_0}^t|u(s)|^{p+1}\text{d}s$ after that $V(t)=0$. The PE guarantees that there exists $t_1\geq t_0$
		when the inequality no longer holds. To estimate this time it is sufficient to find an integer $k$ such that the integral of $|u(t)|^{p+1}$
		from $t_0$ to $t_0+kT$ is greater than $2^{\frac{1-p}{2}}/(1-p)V^{\frac{1-p}{2}}(t_0)$. Using Proposition \ref{PE_2}			
		\begin{align}
			\int_{t_0}^{t_0+kT}|u(s)|^{p+1}\text{d}s\geq kT\epsilon^{p+1}\geq\frac{2^{\frac{1-p}{2}}}{1-p}V^{\frac{1-p}{2}}(t_0).\nonumber
		\end{align}
		Solving for $k$ and taking the least integer that fulfills the inequality yields 		
		\begin{align}
			\label{time1}
			k=\left\lceil \frac{2^{\frac{1-p}{2}}}{(1-p)T\epsilon^{p+1}}V^{\frac{1-p}{2}}(t_0)\right\rceil,
		\end{align}
		then the time that guarantees that $V(t)$ reaches zero is $t_1=k\cdot T+t_0$. Notice that $V(t_0)$ can be replaced by $\frac{1}{2}|x(t_0)|^2$ 
		in \eqref{time1} to obtain \eqref{time0}.
	\subsection{Poof of Theorem \ref{TheoScalar2}}
		Consider again the solution for $V(t)$ in \eqref{SolV}. Since $p>1$, Equation \eqref{SolV} can be rewritten as
		\begin{align}
			V(t)=\frac{1}{\left(V^{\frac{1-p}{2}}(t_0)+2^{\frac{p-1}{2}}(p-1)\int_{t_0}^t|u(s)|^{p+1}\text{d}s\right)^{\frac{2}{p-1}}}.\nonumber
		\end{align}
		As $u(t)$ is of PE the denominator grows unbounded making $V(t)\to 0$ as $t\to\infty$. Now an estimate of the time needed for $V(t)$ to
		decreases from $V(t_0)$ to a value equal or smaller than $c$ is calculated. Substituting $V(t)$ for $c$ in \eqref{dotVa} and evaluating the 
		integral from $t_0$ to $t_1$ it is clear that the value of the integral needs to be larger enough to satisfy the next inequality
		\begin{align}
			\int_{t_0}^{t_1}|u(s)|^{p+1}\text{d}s\geq \frac{1}{2^{\frac{p-1}{2}}(p-1)}\frac{1}{c^{\frac{p-1}{2}}}
				-\frac{1}{V^{\frac{p-1}{2}}(t_0)}.\nonumber
		\end{align}
		Fixing $t_1-t_0$ as an integer multiple of $T$, i.e. $t_1-t_0=kT$, from Proposition \ref{PE_2} we know that the integral is greater or equal to 
		$kT\epsilon^{p+1}$. Forcing the RHS of the last inequality to be less than  $kT\epsilon^{p+1}$ we get
		\begin{align}
			kT\epsilon^{p+1}\geq \frac{1}{2^{\frac{p-1}{2}}(p-1)}\frac{1}{c^{\frac{p-1}{2}}}-\frac{1}{V^{\frac{p-1}{2}}(t_0)}.\nonumber
		\end{align}
		Now solving for $k$ and taking the smallest integer that fulfil the inequality yields			
		\begin{align}
			k&\geq\frac{1}{2^{\frac{p-1}{2}}(p-1)T\epsilon^{p+1}}\left(\frac{1}{c^{\frac{p-1}{2}}}-\frac{1}{V^{\frac{p-1}{2}}(t_0)}\right),\nonumber\\
			k&=\left\lceil\frac{1}{2^{\frac{p-1}{2}}(p-1)T\epsilon^{p+1}}\left(\frac{1}{c^{\frac{p-1}{2}}}
				-\frac{1}{V^{\frac{p-1}{2}}(t_0)}\right)\right\rceil.\nonumber
		\end{align}
		Taking the limit when $V(t_0)\to\infty$ the bound \eqref{time2} is found
		\begin{align}
			k&=\left\lceil \frac{1}{2^{\frac{p-1}{2}}(p-1)T\epsilon^{p+1}}\cdot\frac{1}{c^{\frac{p-1}{2}}}\right\rceil.\nonumber
		\end{align}
	\subsection{Proof of Theorem \ref{TheoScalar3}}
		From \eqref{dotVb} it follows that $\dot{V}(t)\leq -|u^T(t)x(t)|^{p_i+1}=-2^{\frac{p_i+1}{2}}|u(t)|V^{\frac{p_i+1}{2}}(t)$, which define a 
		different differential inequality for each term in the sum. From the Comparison Lemma we known that the solution of $V(t)$ is for below of each 
		solution of $V_{p_i}(t)$, where $\dot{V}_{p_i}(t)=-2^{\frac{p_i+1}{2}}|u(t)|V^{\frac{p_i+1}{2}}_i(t)$; $V_{p_i}(t)$ take the form of \eqref{SolV}. 
		Let $\textbf{P}_M$ and $\textbf{P}_m$ be as in the theorem. For each $p_i\in\textbf{P}_M$ we can assert that $V_{p_i}(t)$ can escape from infinity 
		to a compact set in finite time and so $V(t)$. Fix the level set as $V(x)=1$ and estimate the time needed to reach it
		\begin{align}
			\tau(1)\leq\left\lceil\frac{1}{2^{\frac{p_i-1}{2}}(p_i-1)T\epsilon^{p_i+1}}\right\rceil T,\ p_i\in\textbf{P}_M.\nonumber
		\end{align}
		The smallest time that the algorithm can guarantee is obtained when $(p_i-1)\epsilon^{p_i+1}$ is maximum. Take that quantity as the estimate. 
		Now, with $p_j\in\textbf{P}_m$, we can estimate the time needed for each $V_{p_j}(t)$ to converge from the level set $V(x)=1$ to zero
		\begin{align}
			\textbf{T}(x_0)\leq\left\lceil \frac{1}{(1-p_j)T\epsilon^{p_j+1}}\right\rceil T,\ p_j\in P_m, x_0\big| V(x_0)=1.\nonumber
		\end{align}
		Again, the smallest time that the algorithm can guarantee is when $(1-p_j)\epsilon^{p_j+1}$ is maximized. Then the time needed by the
		algorithm to converge is, at most, the sum of the two estimates.

	\subsection{Proof of Theorem \ref{TVecS}}
		\label{Proof1}
		Following the idea in \cite{AdapSys_Nar_Ann} a lower bound of the integral of $|u^T(t)x(t)|^{p+1}$ is needed. Take the term $|u^T(t)x(t_1)|$
		and add a zero in the form $u^T(t)x(t)-u^T(t)x(t)$ inside the absolute value, by means of the triangle inequality the next partition is obtained
		\begin{align}
			|u^T(t)x(t)|&\geq |u^T(t)x(t_1)|-|u^T(t)\big(x(t_1)-x(t)\big)|.\nonumber
		\end{align}
		Rising both sides to $p+1$ and using the Jensen inequality after that, the inequality yields 		
		\begin{align}
			|u^T(t)x(t)|^{p+1}\geq\frac{1}{2^p}&|u^T(t)x(t_1)|^{p+1}\nonumber\\
				&-|u^T(t)(x(t_1)-x(t))|^{p+1}.\nonumber
		\end{align}
		Integrating both sides from $t_1$ to $t_1+T$ and using Proposition \ref{PE_2} to bound the first term in the RHS becomes		
		\begin{small}		
		\begin{align}
			\int_{t_1}^{t_1+T}|u^T(t)&x(t)|^{p+1}\text{d}t\geq\frac{T\epsilon^{p+1}}{2^p}\|x(t_1)\|^{p+1}\nonumber\\
			\label{Ine01}
				&-\int_{t_1}^{t_1+T}|u^T(t)(x(t_1)-x(t))|^{p+1}\text{d}t.
		\end{align}
		\end{small}
		To bound the magnitude of the second term in the RHS of the last inequality, the next procedure may be used assuming $u(t)$ is uniformly bounded, 
		i.e. $\|u(t)\|\leq u_M$:
		\begin{small}		
		\begin{align}
			\int_{t_1}^{t_1+T}|u^T(s)&(x(t_1)-x(s))|^{p+1}\text{d}s\nonumber\\
				&\leq u_M^{p+1}\int_{t_1}^{t_1+T}\|x(t_1)-x(s)\|^{p+1}\text{d}s\nonumber\\
				&\leq u_M^{p+1}T\sup_{s\in[t_1,t_1+T]}\|x(t_1)-x(s)\|^{p+1}\nonumber\\
				\label{IneIndx}
				&\leq u_M^{p+1}T\Big(\int_{t_1}^{t_1+T}\|\dot{x}(s)\|\text{d}s\Big)^{p+1}.
		\end{align}
		\end{small}
		The norm of $\dot{x}$ in this case is $|u^T(t)x(t)|^p\|u(t)\|$ which is less than $|u^T(t)x(t)|^pu_M$. Hölder inequality can be used with
		$f(t)=|u^T(t)x(t)|^p$ and $g(t)=1$ in the interval $[t_1,t_1+T]$ with $\frac{p+1}{p}$ and $p+1$ as Hölder conjugates, then
		\begin{small}		
		\begin{align}
			T\Big(\int_{t_1}^{t_1+T}|u^T(s)x(s)|^{p+1}&\text{d}s\Big)^p\geq\nonumber\\
				&\left(\int_{t_1}^{t_1+T}|u^T(s)x(s)|^p\text{d}s\right)^{p+1}\nonumber
		\end{align}
		\end{small}
		Using this in \eqref{IneIndx} and then the result in \eqref{Ine01} one gets
		\begin{small}
		\begin{align}
			\int_{t_1}^{t_1+T}|u^T(t)&x(t)|^{p+1}\text{d}t\geq\frac{T\epsilon^{p+1}}{2^p}\|x(t_1)\|^{p+1}\nonumber\\
				&-T^2u_M^{2(p+1)}\left(\int_{t_1}^{t_1+T}|u^T(t)x(t)|^{p+1}\text{d}t\right)^p.\nonumber
		\end{align}
		\end{small}
		This can be rewritten as
		\begin{footnotesize}		
		\begin{align}
			\frac{T\epsilon^{p+1}}{2^p}\|x(t_1)\|^{p+1}\leq& \int_{t_1}^{t_1+T}|u^T(t)x(t)|^{p+1}\text{d}t\nonumber\\
			&+T^2u_M^{2(p+1)}\Big(\int_{t_1}^{t_1+T}|u^T(t)x(t)|^{p+1}\text{d}t\Big)^p,\nonumber
		\end{align}
		\end{footnotesize}
		By algebraic manipulation and defining $\textbf{z}:=\int_{t_1}^{t_1+T}|u^T(t)x(t)|^{p+1}\text{d}t$ the notation can be simplified:
		\begin{align}
			\textbf{z}+T^2u_M^{2(p+1)}\textbf{z}^p\geq\frac{T\epsilon^{p+1}}{2^p}\|x(t_1)\|^{p+1}.\nonumber
		\end{align}
		Notice that the polynomial $P(\textbf{z}):=\textbf{z}+T^2u_M^{2(p+1)}\textbf{z}^p$ is a strict monotonically increasing function for 
		$\textbf{z}\geq 0$ and its inverse exists. By denoting this as $P^{-1}(\textbf{z})$ and using in the inequality above, one gets
		\begin{small}
			\begin{align}
				\textbf{z}=\int_{t_1}^{t_1+T}|u^T(t)x(t)|^{p+1}\text{d}t\geq P^{-1}\left(\frac{T\epsilon^{p+1}}{2^p}\|x(t_1)\|^{p+1}\right).\nonumber
			\end{align}
		\end{small}
		Recalling $\dot{V}(t)$ from \eqref{dotVa} and integrating it for the same time interval it is obvious that
		\begin{align}
			\int_{t_1}^{t_1+T}\dot{V}(t)\text{d}t\leq-P^{-1}\left(\frac{T\epsilon^{p+1}}{2^p}\|x(t_1)\|^{p+1}\right)<0.\nonumber
		\end{align}
		Since $P^{-1}(\cdot)$ is also a strict monotonically increasing function and from Theorem 5 in \cite{Nar_Ann_87} the global uniform asymptotic
		stability of the origin of \eqref{ErrDyn00} is stablished for any $p>0$.

	\subsection{Proof of Theorem \ref{TVecC}}
		In \eqref{dotVb} several terms of the form $|u^T(t)x(t)|^{p_i+1}$ appear; an analysis of each term separately is done before study the full 
		derivative of $V$. By taking the term $|u^T(t)x(t_1)|$ and adding a zero in the form $u^T(t)x(t)-u^T(t)x(t)$ inside the absolute sign and then 
		applying the triangle inequality the next partition is found
		\begin{align}
			|u^T(t)x(t_1)|\leq|u^T(t)(x(t_1)-x(t))|+|u^T(t)x(t)|,\nonumber
		\end{align}
		rising to $p_i+1$ and using the Jensen Inequality		
		\begin{align}
			\frac{1}{2^{p_i}}|u^T(t)x(t_1)|^{p_i+1}\leq&|u^T(t)(x(t_1)-x(t))|^{p_i+1}\nonumber\\
				&+|u^T(t)x(t)|^{p_i+1}.\nonumber
		\end{align}
		Solving for $|u^T(t)x(t)|^{p_i+1}$ and integrating between $[t_1,t_1+T]$
		\begin{footnotesize}
		\begin{align}
			\int_{t_1}^{t_1+T}|u^T(s)x(s)|^{p_i+1}\text{d}s&\geq\frac{1}{2^{p_i}}\int_{t_1}^{t_1+T}|u^T(s)x(t_1)|^{p_i+1}\text{d}s\nonumber\\
				-&\int_{t_1}^{t_1+T}|u^T(s)(x(t_1)-x(s))|^{p_i+1}\text{d}s.\nonumber
		\end{align}
		\end{footnotesize}
		A lower bound of the first term in the RHS of the inequality can be found using Proposition \ref{PE_2}. The following inequality is obtained
		\begin{footnotesize}
		\begin{align}
			\label{IneM}
			\int_{t_1}^{t_1+T}|u^T(s)x(s)|^{p_i+1}\text{d}s&\geq\frac{T\epsilon^{p_i+1}}{2^{p_i}}\|x(t_1)\|^{p_i+1}\nonumber\\
				-&\int_{t_1}^{t_1+T}|u^T(s)(x(t_1)-x(s))|^{p_i+1}\text{d}s.
		\end{align}
		\end{footnotesize}
		A upper bound of the magnitude of the second term is quite more difficult. Assuming $u(t)$ is uniformly bounded by $u_M$, i.e. 
		$\|u(t)\|\leq u_M$, the integral can be bounded as
		\begin{align}
			\int_{t_1}^{t_1+T}|u^T(s)&(x(t_1)-x(s))|^{p_i+1}\text{d}s\nonumber\\
				&\leq u_M^{p_i+1}\int_{t_1}^{t_1+T}\|x(t_1)-x(s)\|^{p_i+1}\text{d}s\nonumber\\
				&\leq u_M^{p_i+1}T\sup_{s\in[t_1,t_1+T]}\|x(t_1)-x(s)\|^{p_i+1}\nonumber\\
			\label{IneNdx}
				&\leq u_M^{p_i+1}T\Big(\int_{t_1}^{t_1+T}\|\dot{x}(s)\|\text{d}s\Big)^{p_i+1}.
		\end{align}
		Now an analysis of the integral of $\|\dot{x}(t)\|$ is required. First the norm of $\dot{x}$ is estimated
		\begin{small}
		\begin{align}
			\|\dot{x}\|&=\Big|\sum_{j=1}^h\lceil u^T(t)x(t)\rfloor^{p_j}\Big|\|u(t)\|
				\leq u_M\sum_{j=1}^h|u^T(t)x(t)|^{p_j}.\nonumber
		\end{align}
		\end{small}
		Integrating and using the Jensen inequality for the function $(\cdot)^{p_i+1}$	
		\begin{footnotesize}
		\begin{align}
			\label{IneNdx1}
			\Big(\int_{t_1}^{t_1+T}\|\dot{x}(s)&\|\text{d}s\Big)^{p_i+1}\leq\nonumber\\
				&u_M^{p_i+1}h^{p_i}\sum_{j=1}^h\left(\int_{t_1}^{t_1+T}|u^T(s)x(s)|^{p_j}\text{d}s\right)^{p_i+1}.
		\end{align}
		\end{footnotesize}
		The restriction $p_M\leq p_m+1$ guarantees that $\frac{p_i+1}{p_j}\geq 1$ for every $(i,j)$ and then 
		$\frac{p_i+1}{p_j}$, $\frac{p_i+1}{p_i+1-p_j}$ are Hölder conjugates. Now the Hölder inequality is used with $f(t)=|u^T(t)x(t)|^{p_j}$ and 
		$g(t)=1$ in the interval $[t_1,t_1+T]$ and Hölder conjugates as described before	
		\begin{footnotesize}
		\begin{align}
			\label{IneNdx2}
			T^{p_i-p_j+1}\Big(\int_{t_1}^{t_1+T}|u^T(s)x(&s)|^{p_i+1}\text{d}s\Big)^{p_j}\geq \nonumber\\
				&\Big(\int_{t_1}^{t_1+T}|u^T(s)x(s)|^{p_j}\text{d}s\Big)^{p_i+1}.
		\end{align}
		\end{footnotesize}
		Using \eqref{IneNdx} in \eqref{IneNdx1}, applying \eqref{IneNdx2} in the sum, one has
		\begin{footnotesize}
		\begin{align}
			\label{IneNdx3}
			u_M^{2(p_i+1)}h^{p_i}T^{p_i+2}\sum_{j=1}^{h}\frac{1}{T^{p_j}}&\Big(\int_{t_1}^{t_1+T}|u^T(s)x(s)|^{p_i+1}\text{d}s\Big)^{p_j}\nonumber\\
			\geq& u_M^{p_i+1}T\Big(\int_{t_1}^{t_1+T}\|\dot{x}(s)\|\text{d}s\Big)^{p_i+1}.				
		\end{align}
		\end{footnotesize}
		From \eqref{IneM}, \eqref{IneNdx} and \eqref{IneNdx3} we obtain
		\begin{footnotesize}
		\begin{align}
			\int_{t_1}^{t_1+T}&|u^T(s)x(s)|^{p_i+1}\text{d}s\geq\frac{T\epsilon^{p_i+1}}{2^{p_i}}\|x(t_1)\|^{p_i+1}\nonumber\\
				-&u_M^{2(p_i+1)}h^{p_i}T^{p_i+2}\sum_{j=1}^{h}\frac{1}{T^{p_j}}\Big(\int_{t_1}^{t_1+T}|u^T(s)x(s)|^{p_i+1}\text{d}s\Big)^{p_j}.\nonumber
		\end{align}
		\end{footnotesize}
		Defining $\textbf{z}_i:=\int_{t_1}^{t_1+T}|u^T(s)x(s)|^{p_i+1}\text{d}s$ the last inequality can be rewritten as
		\begin{small}
		\begin{align}
			\textbf{z}_i+u_M^{2(p_i+1)}h^{p_i}T^{p_i+2}\sum_{j=1}^{h}\frac{1}{T^{p_j}}\textbf{z}_i^{p_j}\geq 
				\frac{T\epsilon^{p_i+1}}{2^{p_i}}\|x(t_1)\|^{p_i+1}.\nonumber
		\end{align}
		\end{small}
		The polynomial in the LHS is a strict monotonically increasing function and zero when $\textbf{z}_i=0$, then its inverse exists and also is a
		strict monotonically increasing function. By denoting the polynomial as $P_i$ and its inverse as $P_i^{-1}$ one gets	
		\begin{footnotesize}		
		\begin{align}
			\label{InePol}
			\int_{t_1}^{t_1+T}|u^T(s)x(s)|^{p_i+1}\text{d}s\geq P_i^{-1}\left(\frac{T\epsilon^{p_i+1}}{2^{p_i}}\|x(t_1)\|^{p_i+1}\right).
		\end{align}
		\end{footnotesize}
		Now integrating \eqref{dotVb} from $t_1$ to $t_1+T$ and using \eqref{InePol} the next inequality for $\dot{V}$ is obtained
		\begin{small}		
		\begin{align}
			\int_{t_1}^{t_1+T}\dot{V}(s)\text{d}s&=-\sum_{i=1}^{h}\int_{t_1}^{t_1+T}|u^T(s)x(s)|^{p_i+1}\text{d}s\nonumber\\
					&\leq -\sum_{i=1}^hP_i^{-1}\left(\frac{T\epsilon^{p_i+1}}{2^{p_i}}\|x(t_1)\|^{p_i+1}\right).\nonumber
		\end{align}
		\end{small}
		This satisfies the assumptions of the Theorem 5 in \cite{Nar_Ann_87} and guarantees the global uniform asymptotic stability of the 
		origin of \eqref{ErrDyn01}.

\end{document}